\documentclass[12pt]{amsart}
\usepackage{amssymb}

\newtheorem{theorem}{Theorem}[section]
\newtheorem{corollary}{Corollary}[section]
\newtheorem{proposition}{Proposition}[section]
\newtheorem{definition}{Definition}[section]

\newcommand\relphantom[1]{\mathrel{\phantom{#1}}}

\DeclareMathOperator\spec{spectrum}

\title{On universality and convergence of the Fourier series of functions in the disc algebra.}
\author{C. Papachristodoulos, M. Papadimitrakis.}
\address{Department of Mathematics and Applied Mathematics, University of Crete, 70013 Iraklio, Crete, Greece}
\email{papadim@math.uoc.gr}
\thanks{The first author is supported by the project 3083-FOURIERDIG which is implemented under the ``Aristeia II'' Action of the ``Operational Programme 
Education and Lifelong Learning'' and is co-founded by the European Social Fund (ESF) and National Resources.}

\begin{document}
{\allowdisplaybreaks
\pagestyle{plain}
\begin{abstract}
We construct functions in the disc algebra with pointwise universal Fourier series on sets which are $G_{\delta}$ and dense and at the same time with Fourier 
series whose set of divergence is of Hausdorff dimension zero. We also see that some classes of closed sets of measure zero do not accept uniformly universal 
Fourier series, although all such sets accept divergent Fourier series.  
\end{abstract}
\maketitle

\section{Introduction and notation.}

\par Let $\mathbb D=\{z\in\mathbb C\,|\,|z|<1\}$, $\mathbb T=\{z\in\mathbb C\,|\,|z|=1\}=\mathbb R/2\pi$. We denote by $C(\mathbb T)$ the set of complex
continuous functions with the supremum norm $\|\cdot\|$ and, for $f\in C(\mathbb T)$, by $S_n(f,t)$ the $n$-th partial sum of the Fourier series of $f$ at the
point $t\in\mathbb T$,
$$ S_n(f,t)=\sum_{k=-n}^n\widehat f(k)e^{ikt},$$
where
$$\widehat f(k)=\frac 1{2\pi}\int_{-\pi}^{\pi}f(t)e^{-ikt}\,dt,\qquad k\in\mathbb Z,$$
is the $k$-th Fourier coefficient of $f$. 
\par Also, let $A(\mathbb D)=\{f\in C(\mathbb T)\,|\,\widehat f(k)=0\,\,\text{for}\,\,k<0\}$ be the disc algebra with the supremum norm.
\par If $(X,d)$ is a complete metric space, a property is said to be satisfied at quasi all points of $X$ if it is satisfied at a $G_{\delta}$ and dense set,
i.e. at a topologically large set.
\par First we recall a few facts regarding the divergence of the partial sums $S_n(f,t)$ that the reader must have in mind.
\par The first, from \cite{K}, is that quasi all $f\in C(\mathbb T)$ have the property that their Fourier series diverge at quasi all points of $\mathbb T$.
\par The second (see \cite{Ka}) is the classical result that $E\subseteq\mathbb T$ is a set of divergence for $C(\mathbb T)$, i.e. there is a continuous 
function whose Fourier series diverges at all points of $E$, if and only if $E$ is a set of infinite divergence for $C(\mathbb T)$, i.e. there is a continuous 
function $f$ such that $\varlimsup|S_n(f,t)|=+\infty$ for all $t\in E$. Moreover, every set of Lebesgue measure zero is a set of divergence for $C(\mathbb T)$.
\par It is not hard to see that a set $E\subseteq\mathbb T$ may be $G_{\delta}$ and dense and simultaneously of Lebesgue measure zero (see \cite{Ox}). Of
course, by Carleson's theorem a set of divergence for $C(\mathbb T)$ has necessarily measure zero.
\par The third fact, from \cite{BH}, is that the set $\{t\,|\,\varlimsup|S_n(f,t)|=+\infty\}$ has Hausdorff dimension equal to $1$ for quasi all $f\in 
C(\mathbb T)$. 
\par Recently, in \cite{M} and \cite{HK}, a different notion of divergence has been studied. We present the definitions and the basic results of these papers. 
\begin{definition}
Let $E\subseteq\mathbb T$. We say that $f\in C(\mathbb T)$ is pointwise universal on $E$ if for every $g:E\to\mathbb C$ belonging to the Baire-1 class there 
exists a strictly increasing sequence of positive integers $(k_n)$ such that $S_{k_n}(f,t)\to g(t)$ for all $t\in E$.\newline
We denote the class of these functions by $U_p(E)$.
\end{definition}
\begin{definition}
Let $K\subseteq\mathbb T$ be a compact set. We say that $f\in C(\mathbb T)$ is uniformly universal on $K$ if for every continuous $g:K\to\mathbb C$ there 
exists a strictly increasing sequence of positive integers $(k_n)$ such that $\|S_{k_n}(f,\cdot)-g\|_K\to 0$, where $\|\cdot\|_K$ is the supremum norm on 
$K$.\newline
We denote the class of these functions by $U(K)$.
\end{definition}
\par Of course the notions of pointwise universality and uniform universality coincide when the set $E=K$ is finite. In this case we speak about universality 
on $E=K$.
\par If we denote by $\mathcal K(\mathbb T)$ the complete metric space of all nonempty compact subsets of $\mathbb T$ with the Hausdorff metric then we know 
that quasi all $f\in C(\mathbb T)$ and quasi all $f\in A(\mathbb D)$ are uniformly universal on quasi all sets $K\in\mathcal K(\mathbb T)$. (For $C(\mathbb 
T)$ see \cite{M} and for $A(\mathbb D)$ see \cite{HK}.) And since quasi all sets in $\mathcal K(\mathbb T)$ are perfect sets (see \cite{Ko}) we know that there 
are perfect $K\subseteq\mathbb T$ such that $A(\mathbb D)\cap U(K)\neq\emptyset$.
\par Finally, we know that for each countable $E\subseteq\mathbb T$ quasi all $f\in C(\mathbb T)$ and quasi all $f\in A(\mathbb D)$ are pointwise universal on 
$E$. (For $C(\mathbb T)$ see \cite{M} and for $A(\mathbb D)$ see \cite{HK}.)
\par The proofs of the above results are not constructive. They use Baire's category theorem. Hence a first question which arises is to construct a uniformly or 
pointwise universal function. A second question is what can we say about the convergence of the Fourier series outside $E$ of a function in $U_p(E)$. If for 
example $f\in A(\mathbb D)\cap U_p(E)$, where $E$ is a countable dense set in $\mathbb T$, then the Fourier series of $f$ cannot converge at all points outside 
$E$, since $E\subseteq G\subseteq D$, where $G=\bigcap_{N=1}^{+\infty}\bigcup_{n=1}^{+\infty}\{t\,|\,S_n(f,t)>N\}$ and 
$D=\{t\,|\,S_n(f,t)\,\,\text{diverges}\}$, and $G$ is $G_{\delta}$ and dense and hence uncountable. Also, taking into account that the pointwise universal 
functions are highly divergent, we may ask whether it is possible that the above set $D$ has Hausdorff dimension less than $1$. We deal with these questions in 
sections 2 and 3. More precisely, we give a method to construct pointwise universal functions in $A(\mathbb D)$ on finite and countably infinite sets. We also 
give a criterion for convergence of the Fourier series outside the set of pointwise universality and we see that the above set $D$ can even be of Hausdorff 
dimension equal to $0$. Of course, the above method can also be applied for functions in $C(\mathbb T)$.
\par In section 4 we turn to the study of uniform universality (see definition 1.2). By Carleson's theorem the perfect sets which accept uniform universality 
must have Lebesgue measure zero. Of course the first class of such perfect sets that comes to mind consists of the familiar Cantor type sets. We prove that 
these sets do not accept uniform universality. Moreover, we prove that the same is true for a class of compact countable sets. This is in contrast with the 
well known fact that all sets of measure zero are sets of divergence for $C(\mathbb T)$.
\par Finally, we close this paper with some open problems.
\par In the following the symbol $C$ will denote an absolute constant which may change from one relation to the next. 

\section{Universality on finite $K\subseteq\mathbb T$ and convergence on $\mathbb T\setminus K$.}

\par Let $N>n$. We consider the Fejer polynomials
$$Q_{N,n}(t)=2\sin Nt\sum_{k=1}^n\frac{\sin kt}k=\sum_{m=-(N+n)}^{N+n}\widehat{Q_{N,n}}(m)e^{imt},$$
where
\begin{equation}
\widehat{Q_{N,n}}(\pm m)=\begin{cases}\frac 1{2k}, &\text{$m=N-k$, $k=1,\ldots,n$}\notag\\
-\frac 1{2k}, &\text{$m=N+k$, $k=1,\ldots,n$}\\
0, &\text{otherwise}\end{cases}
\end{equation}
\par Obviously,
$$Q_{N,n}(0)=0.$$
\par It is well known that the Fejer polynomials are uniformly bounded, i.e.
$$\|Q_{N,n}\|\leq C.$$
\par Also,
$$S_N(Q_{N,n},0)=\frac 1n+\cdots+1\sim\log n$$
and the sequences of the Fourier coefficients of $Q_{N,n}$ are of uniform bounded variation, i.e.
$$\sum_{m=-(N+n)}^{N+n+1}|\widehat{Q_{N,n}}(m-1)-\widehat{Q_{N,n}}(m)|\leq C.$$
\par Finally, we have the estimate
$$|S_k(Q_{N,n},t)|\leq\frac C{|t|}\qquad\text{for all}\,\,k\,\,\text{and}\,\,t\neq 0.$$
\par For all these properties see \cite{Ka}, Ch II, exercise 2.3.
\par Now, considering arbitrary $c\in\mathbb C$ and $\epsilon>0$ and setting
$$P_{N,n}=\frac c{S_N(Q_{N,n},0)}\,Q_{N,n},$$
it follows easily that if $N>n$ are large enough then
$$P_{N,n}(0)=0,$$
$$S_{N}(P_{N,n},0)=c,$$
$$\|P_{N,n}\|<\epsilon,$$ 
$$\sum_{m=-(N+n)}^{N+n+1}|\widehat{P_{N,n}}(m-1)-\widehat{P_{N,n}}(m)|\leq\epsilon,$$
$$|S_k(P_{N,n},t)|\leq\frac{\epsilon}{|t|}\qquad\text{for all}\,\,k\,\,\text{and}\,\,t\neq 0.$$
\par Let $\{c_j\,|\,j\in\mathbb N\}$ be a countable dense set in $\mathbb C$ and let $\epsilon, \epsilon_j>0$ with
$$\epsilon=\sum_{j=1}^{+\infty}\epsilon_j.$$
\par Now for each $j$ we can choose arbitrarily large $N_j>n_j$ and polynomials $P_{N_j,n_j}$ such that
\begin{equation}
P_{N_j,n_j}(0)=0, 
\end{equation}
\begin{equation}
 S_N(P_{N_j,n_j},0)=c_j,
\end{equation}
\begin{equation}
 \|P_{N_j,n_j}\|<\epsilon_j,
\end{equation}
\begin{equation}
 \sum_{m=-(N_j+n_j)}^{N_j+n_j+1}|\widehat{P_{N_j,n_j}}(m-1)-\widehat{P_{N_j,n_j}}(m)|\leq\epsilon_j.
\end{equation}
\begin{equation}
 |S_k(P_{N_j,n_j},t)|\leq\frac{\epsilon_j}{|t|}\qquad\text{for all}\,\,k\,\,\text{and}\,\,t\neq 0.
\end{equation}
\par We also choose the $N_j>n_j$ to satisfy the inequalities
\begin{equation}
2N_j+(N_j+n_j)<2N_{j+1}-(N_{j+1}+n_{j+1}) 
\end{equation}
for all $j$ and we set
\begin{equation}
\begin{split}
P_j(t)&=e^{2iN_jt}P_{N_j,n_j}(t)\\
\mathcal B_j&=\{N_j-n_j,\ldots,3N_j+n_j\}\supseteq\spec(P_j).
\end{split}
\end{equation}
\par From (6) we have
\begin{equation}
\mathcal B_j\prec \mathcal B_{j+1}, 
\end{equation}
where $\prec$ means that the block $\mathcal B_j$ lies to the left of $\mathcal B_{j+1}$, that is $\max\mathcal B_j<\min\mathcal B_{j+1}$.
\par Hence (3) implies that the series
$$\sum_{j=1}^{+\infty}P_j=g$$
converges uniformly to a function $g\in C(\mathbb T)$ such that
$$\|g\|<\epsilon.$$
Also, from (7) and (8) we get 
\begin{equation}
\begin{split}
\widehat{g}(n)=\begin{cases}\widehat{P_j}(n)=\widehat{P_{N_j,n_j}}(n-2N_j), &\text{if $n\in\mathcal B_j$, $j\in\mathbb N$}\\ 0, &\text{otherwise}\end{cases}
\end{split}
\end{equation}
In particular $\widehat{g}(n)=0$ for $n<0$ and thus $g$ belongs to $A(\mathbb D)$.
\par Also, (1) and (2) imply
\begin{equation}
S_{3N_j}(g,0)=c_j.
\end{equation}
Consequently, $g$ is universal on $K=\{0\}$.
\par Moreover, (4), (8) and (9) imply
$$\sum_{n=0}^{+\infty}|\widehat{g}(n-1)-\widehat{g}(n)|<\epsilon$$
i.e. the sequence $(\widehat{g}(n))$ is of bounded variation. Hence the Fourier series $\sum_{n=0}^{+\infty}\widehat{g}(n)e^{int}$ converges for $t\neq 0$
and uniformly in each closed interval of $\mathbb T$ which does not contain $0$. See \cite{Z}, Ch I, Theorem (2.6). Moreover, (5) implies
\begin{equation}
|S_k(P_j,t)|\leq\frac{\epsilon_j}{|t|}\qquad\text{for all}\,\,j,k\,\,\text{and}\,\,t\neq 0.
\end{equation}
\par From (7) and (8) and the uniform convergence of the series $\sum_{j=1}^{+\infty}P_j$ we get
$$S_n(g,t)\to g(t)\qquad\text{as}\,\,n\to+\infty\,\,\text{and}\,\,n\notin\bigcup_{j=1}^{+\infty}\mathcal B_j.$$
\par The previous constructions can be extended for any finite number of points $t_1,\ldots,t_m\in\mathbb T$. For simplicity we present the construction for two
points.
\par Let $t_1\neq t_2$ and $\{(a_j,b_j)\,|\,j\in\mathbb N\}$ be a countable dense set in $\mathbb C^2$. We consider the functions 
$$f_1(t)=g_1(t-t_1)=\sum_{j=1}^{+\infty}P_j^{(1)}(t),\quad f_2(t)=g_2(t-t_2)=\sum_{j=1}^{+\infty}P_j^{(2)}(t),$$ 
where $g_1$ is the function $g$ constructed above with $(a_j)$ in place of $(c_j)$ and $g_2$ is the function $g$ with $(b_j)$ in place of $(c_j)$. Note that 
the polynomials $P_j^{(1)},P_j^{(2)}$ have their spectrum in $\mathcal B_j$ and satisfy (11) with $t-t_1,t-t_2$ repsectively at the denominator of the right 
side.
\par Relation (10) becomes
$$S_{3N_j}(f_1,t_1)=a_j,\qquad S_{3N_j}(f_2,t_2)=b_j.$$
\par Note that, since $N_j, n_j$ in the previous constructions can be chosen arbitrarily large, we may consider them to be the same for the functions $g_1,g_2$.
\par Then 
$$f_1\in A(\mathbb D)\cap U(\{t_1\}),\qquad f_2\in A(\mathbb D)\cap U(\{t_2\}).$$
\par Also $S_n(f_1,t)\to f_1(t)$ for $t\neq t_1$ and $S_n(f_2,t)\to f_2(t)$ for $t\neq t_2$ and uniformly on every closed interval in $\mathbb T$ which 
does not contain $t_1,t_2$ and $S_n(f_1,t)\to f_1(t)$ and $S_n(f_2,t)\to f_2(t)$ for every $t$ as $n\to+\infty$ and $n\notin\bigcup_{j=1}^{+\infty}\mathcal 
B_j$.
\par Taking into account that the set $\{(a_j+f_2(t_1),b_j+f_1(t_2))\}$ is dense in $\mathbb C^2$ we get the following for the function $f=f_1+f_2$ or, more 
generally, for the function $f=f_1+\cdots+f_m$ when $K=\{t_1,\ldots,t_m\}$.
\begin{theorem}
Let $K=\{t_1,\ldots,t_m\}$ be a finite set in $\mathbb T$. For every $\epsilon_j>0$ with $\sum_{j=1}^{+\infty}\epsilon_j<+\infty$ there are blocks $\mathcal 
B_j$ in $\mathbb N$ such that $\mathcal B_1\prec\mathcal B_2\prec\ldots\,$ and corresponding polynomials $P_j$ with $\spec(P_j)\subseteq\mathcal B_j$ and 
$\|P_j\|<\epsilon_j$ so that the function $f=\sum_{j=1}^{+\infty}P_j$ has $\spec(f)\subseteq\bigcup_{j=1}^{+\infty}\mathcal B_j$ and the following 
properties:\newline
$(i)$ $f\in A(\mathbb D)\cap U(K)$,\newline
$(ii)$ $\|f\|\leq\epsilon=\sum_{j=1}^{+\infty}\epsilon_j$,\newline
$(iii)$ $S_n(f,t)\to f(t)$ for every $t\notin K$ and uniformly on every closed interval in $\mathbb T$ which does not intersect $K$.\newline
$(iv)$ $|S_k(P_j,t)|\leq\epsilon_j\sum_{l=1}^m\frac 1{|t-t_l|}$ for all $j,k$ and $t\neq t_1,\ldots,t_m$.\newline
$(v)$ $S_n(f,t)\to f(t)$ for every $t$ as $n\to+\infty$ and $n\notin\bigcup_{j=1}^{+\infty}\mathcal B_j$.
\end{theorem}
\par We note that the $\min\mathcal B_j$ can be taken arbitrarily large.
\par

\section{Pointwise universality on countably infinite $E\subseteq\mathbb T$ and convergence on $\mathbb T\setminus E$.}

\par Let $E=\{t_l\,|\,l\in\mathbb N\}\subseteq\mathbb T$ be a countably infinite set. We begin with the construction of a function $f\in A(\mathbb D)\cap
U_p(E)$.
\par Let $E_m=\{t_1,\ldots,t_m\}$. By Theorem 2.1 we have that for each $\epsilon_{m,j}>0$ with $\sum_{j=1}^{+\infty}\epsilon_{m,j}<+\infty$ there are blocks 
$\mathcal B_{m,j}$ in $\mathbb N$ such that $\mathcal B_{m,1}\prec\mathcal B_{m,2}\prec\ldots$ and corresponding polynomials $P_{m,j}$ with 
$\spec(P_{m,j})\subseteq\mathcal B_{m,j}$ and $\|P_{m,j}\|<\epsilon_{m,j}$ so that the 
function
$$f_m=\sum_{j=1}^{+\infty}P_{m,j}$$
belongs to $f\in A(\mathbb D)\cap U(E_m)$ and satisfies
\begin{equation}
 \|f_m\|\leq\epsilon_m=\sum_{j=1}^{+\infty}\epsilon_{m,j}
\end{equation}
and
\begin{equation}
 |S_k(P_{m,j},t)|\leq\epsilon_{m,j}\sum_{l=1}^m\frac 1{|t-t_l|}\qquad\text{for all}\,\,k,j\,\,\text{and}\,\,t\neq t_1,\ldots,t_m. 
\end{equation}
\par Since $\min\mathcal B_{m,j}$ can be taken arbitrarily large, we may take $\mathcal B_{m,j}$ in the following diagonal order:
$$\mathcal B_{1,1}\prec\mathcal B_{2,1}\prec\mathcal B_{1,2}\prec\mathcal B_{3,1}\prec\mathcal B_{2,2}\prec\mathcal B_{1,3}\prec\ldots\,.$$
We may also assume that
\begin{equation}
 \sum_{m=1}^{+\infty}\epsilon_m<+\infty.
\end{equation}
\par We set
$$f=\sum_{m=1}^{+\infty}f_m=\sum_{m=1}^{+\infty}\sum_{j=1}^{+\infty}P_{m,j}.$$
\par By (12), (14) it follows that $f\in A(\mathbb D)$. We now prove that $f$ is pointwise universal on $E$.
\begin{theorem}
Let $f$ be the function constructed above.\newline
(a) $f\in A(\mathbb D)\cap U_p(E)$.\newline
(b) For each $t\in\mathbb T\setminus E$ satisfying the condition
\begin{equation}
\max_j\epsilon_{m,j}\sum_{l=1}^m\frac 1{|t-t_l|}\to 0 
\end{equation}
we have
$$ S_n(f,t)\to f(t).$$
In particular, if $m\max_j\epsilon_{m,j}\to 0$ and the distance of $t$ from $E$ is positive, then $S_n(f,t)\to f(t)$.
\end{theorem}
\begin{proof}
(a) Let $h:E\to\mathbb C$ be an arbitrary function. 
\par We first choose $m_1$ such that
$$\sum_{m=m_1+1}^{+\infty}\epsilon_m<\frac{\delta_1}3,$$
where $\delta_1=1$.
\par From the fact that the blocks $\mathcal B_{m,j}$ are mutually disjoint and from Theorem 2.1, we get 
\begin{equation}
\begin{split}
S_n(f_m,t_l)\to f_m(t_l)\qquad\text{as}\,\,n&\to+\infty,\,\,n\in\bigcup_{j=1}^{+\infty}\mathcal B_{m_1,j},\\
&1\leq l\leq m_1,\,\,1\leq m\leq m_1-1. 
\end{split}
\end{equation}
\par From (12) it follows that
\begin{equation}
\begin{split}
\sum_{m=m_1+1}^{+\infty}|S_n(f_m,t_l)|\leq\sum_{m=m_1+1}^{+\infty}&\sum_{j=1}^{+\infty}\epsilon_{m,j}<\frac{\delta_1}3,\\
&1\leq l\leq m_1,\,\,n\in\bigcup_{j=1}^{+\infty}\mathcal B_{m_1,j}. 
\end{split}
\end{equation}
\par Also, from the universality of $f_{m_1}$ on $E_{m_1}=\{t_1,\ldots,t_{m_1}\}$ and from (15) we get that there exists $n_1\in\bigcup_{j=1}^{+\infty}\mathcal 
B_{m_1,j}$ so that
\begin{equation}
\Big|S_{n_1}(f_{m_1},t_l)-\Big(h(t_l)-\sum_{m=1}^{m_1-1}f_m(t_l)\Big)\Big|<\frac{\delta_1}3,\qquad 1\leq l\leq m_1 
\end{equation}
\begin{equation}
\begin{split}
\big|S_{n_1}(f_m,t_l)-f_m(t_l)\big|<&\frac{\delta_1}{3(m_1-1)},\\
& 1\leq l\leq m_1,\,\,1\leq m\leq m_1-1. 
\end{split}
\end{equation}
\par Now we observe that the $n_1$-th Fourier sum of $f$ is a finite sum of $n_1$-th Fourier sums of the functions $f_1,f_2,\ldots,f_{m_1'}$ for some $m_1'\geq
m_1$. Hence
\begin{equation}
\begin{split}
S_{n_1}(f,t_l)&-h(t_l)\notag\\
&=\big(S_{n_1}(f_1,t_l)-f_1(t_l)\big)+\cdots+\big(S_{n_1}(f_{m_1-1},t_l)-f_{m_1-1}(t_l)\big)\notag\\
&\relphantom{=}{}+\Big(S_{n_1}(f_{m_1},t_l)-\Big(h(t_l)-\sum_{m=1}^{m_1-1}f_m(t_l)\Big)\Big)\notag\\
&\relphantom{=}{}+S_{n_1}(f_{m_1+1},t_l)+\cdots+S_{n_1}(f_{m_1'},t_l).
\end{split}
\end{equation}
\par Finally, from (17), (18), (19) we get
$$|S_{n_1}(f,t_l)-h(t_l)|<\delta_1,\qquad 1\leq l\leq m_1.$$
\par Similarly, for $\delta_2=\frac 12$ there exists $m_2>m_1$ such that $\sum_{m=m_2+1}^{+\infty}\epsilon_m<\frac{\delta_2}3$ and there exists $n_2>n_1$, 
$n_2\in\bigcup_{j=1}^{+\infty}\mathcal B_{m_2,j}$ such that
$$|S_{n_2}(f,t_l)-h(t_l)|<\delta_2,\qquad 1\leq l\leq m_2.$$
\par Continuing in this manner, we construct strictly increasing sequences of positive integers $(m_N)$, $(n_N)$ such that
$$|S_{n_N}(f,t_l)-h(t_l)|<\delta_N=\frac 1N,\qquad 1\leq l\leq m_N.$$
\par This implies that
$$S_{n_N}(f,t)\to h(t),\qquad t\in E $$
and the proof of pointwise universality is complete.\newline
(b) By our diagonal ordering of the blocks $\mathcal B_{m,j}$ every $n\in\mathbb N$ lies in some diagonal $\mathcal B_{k,1},\mathcal 
B_{k-1,2},\ldots,\mathcal B_{1,k}$. I.e. $\max\mathcal B_{1,k-1}<n\leq\max\mathcal B_{k,1}$ or $\max\mathcal B_{k-j+1,j}<n\leq\max\mathcal B_{k-j,j+1}$ 
for some $j=1,2,\ldots,k-1$. 
\par Hence by the definition of $f$ we have
\begin{equation}
\begin{split}
S_n(f,t)&=P_{1,1}(t)+P_{2,1}(t)+P_{1,2}(t)+\cdots\notag\\
&\relphantom{=}{}+P_{k,1}(t)+P_{k-1,2}(t)+\cdots+P_{k-j+1,j}(t)+S_n(P_{k-j,j+1},t)
\end{split}
\end{equation}
if $n\in\mathcal B_{k-j,j+1}$. Note that the final term in the last sum is missing in case $\max\mathcal B_{k-j+1,j}\leq n<\min\mathcal B_{k-j,j+1}$.
\par Similarly, if $m>n$ then $m$ lies in some diagonal $\mathcal B_{p,1},\mathcal B_{p-1,2},\ldots,\mathcal B_{1,p}$ with $p\geq k$.
\par Therefore,
\begin{equation}
\begin{split}
|S_n(f,t)-S_m(f,t)|&\leq|P_{k-j,j+1}(t)-S_n(P_{k-j,j+1},t)|\notag\\
&\relphantom{\leq}{}+\big(|P_{k-j-1,j+2}(t)|+\cdots+|P_{p-q+1,q}(t)|\big)\notag\\
&\relphantom{\leq}{}+|S_m(P_{p-q,q+1},t)|
\end{split}
\end{equation}
\par Now (13) implies
$$|S_m(P_{p-q,q+1},t)|\leq\epsilon_{p-q,q+1}\sum_{l=1}^{p-q}\frac 1{|t-t_l|}$$
and, together with (12), 
$$|P_{k-j,j+1}(t)-S_n(P_{k-j,j+1},t)|\leq\epsilon_{k-j,j+1}+\epsilon_{k-j,j+1}\sum_{l=1}^{k-j}\frac 1{|t-t_l|}.$$
\par Hence
\begin{equation}
\begin{split}
|S_n(f,t)-S_m(f,t)|&\leq\epsilon_{k-j,j+1}+\epsilon_{k-j,j+1}\sum_{l=1}^{k-j}\frac 1{|t-t_l|}\notag\\
&\relphantom{\leq}{}+\epsilon_{k-j-1,j+2}+\cdots+\epsilon_{p-q+1,q}\notag\\
&\relphantom{\leq}{}+\epsilon_{p-q,q+1}\sum_{l=1}^{p-q}\frac 1{|t-t_l|}
\end{split}
\end{equation}
\par If $n,m\to+\infty$ then $k,p\to+\infty$ and since the double series $\sum_{r,s}\epsilon_{r,s}$ converges, we get 
$\epsilon_{k-j,j+1}+\epsilon_{k-j-1,j+2}+\cdots+\epsilon_{p-q+1,q}\to 0$.  
\par Regarding the term $r=\epsilon_{k-j,j+1}\sum_{l=1}^{k-j}\frac 1{|t-t_l|}$ we have two cases. If $k-j$ is bounded then $j+1\to+\infty$ and thus $r\to 0$. 
If $k-j\to+\infty$ then (15) implies $r\to 0$. Similarly, $\epsilon_{p-q,q+1}\sum_{l=1}^{p-q}\frac 1{|t-t_l|}\to 0$.
\par We conclude that $(S_n(f,t))$ converges. The proof will be complete if we show that $S_n(f,t)\to f(t)$ when $n\to+\infty$ and does not belong to any of 
the blocks $\mathcal B_{m,j}$. Indeed, in this case $(v)$ of Theorem 2.1 implies
$$S_n(f,t)=\sum_{m=1}^{+\infty}S_n(f_m,t)\to\sum_{m=1}^{+\infty}f_m(t)=f(t)$$
due to the uniform bound $\sum_{m=1}^{+\infty}|S_n(f_m,t)|\leq\sum_{m=1}^{+\infty}\epsilon_m$. 
\end{proof}
\par As an application we get the following for the function $f$ of Theorem 3.1.
\begin{theorem}
Let $(\delta_m)$ be a decreasing sequence of positive numbers such that $\sum_{m=1}^{+\infty}\delta_m^a<+\infty$ for all $a>0$. Now, if 
$\max_j\epsilon_{m,j}\sum_{l=1}^m\frac 1{\delta_l}\to 0$, then we have $S_n(f,t)\to f(t)$ outside a set of Hausdorff dimension zero.
\end{theorem}
\begin{proof}
We consider
$$I_m=(t_m-\delta_m,t_m+\delta_m),\qquad D=\bigcap_{m=1}^{+\infty}\bigcup_{l=m}^{+\infty}I_l.$$
\par Since $\sum_{m=1}^{+\infty}\delta_m^a<+\infty$ for all $a>0$, the set $D$ is of Hausdorff dimension zero.
\par Now, if $t$ is not in the countable set $E$ neither in $D$, then it satisfies (15). Indeed, let $t\notin E$ and $t\notin\bigcup_{l=m_0}^{+\infty}I_l$ for 
some $m_0$. Then if $m\geq m_0$ we have
$$\max_j\epsilon_{m,j}\sum_{l=1}^m\frac 1{|t-t_l|}\leq\max_j\epsilon_{m,j}\sum_{l=1}^{m_0-1}\frac 1{|t-t_l|}+\max_j\epsilon_{m,j}\sum_{l=m_0}^m\frac 
1{|t-t_l|}.$$
\par The first term of the right side obviously tends to $0$ as $m\to+\infty$. As for the second term we have
$$\max_j\epsilon_{m,j}\sum_{l=m_0}^m\frac 1{|t-t_l|}\leq\max_j\epsilon_{m,j}\sum_{l=m_0}^m\frac 1{\delta_l}\leq\max_j\epsilon_{m,j}\sum_{l=1}^m\frac 
1{\delta_l}$$
which by our hypothesis tends to $0$ as $m\to+\infty$.
\end{proof}
\par Note that, if $E$ is dense in $\mathbb T$, the set of divergence of $S_n(f,t)$ is necessarily $G_{\delta}$ and dense and hence uncountable, although by 
the proper choice of $f$ it can be of Hausdorff dimension zero.

\section{Subsets of $\mathbb T$ not accepting uniform universality.}

\par Let $K$ be a compact subset of $\mathbb T$, $K\neq\mathbb T$. In Theorem 2.2 of \cite{KNP} and in \cite{N} it was shown that $H(\mathbb D)\cap 
U(K)$ is a dense-$G_{\delta}$ subset of $H(\mathbb D)$ when the latter has the topology of uniform convergence on compacta.
\par We now prove a special case of Theorem 4.3 of \cite{KNP} replacing the condition appearing there with a much simpler one (number (20) in what follows).
\begin{proposition}
(a) Let $K$ be a compact proper subset of $\mathbb T$ and $t_0\in K$ with the following property:
\begin{equation}
\begin{split}
&\text{for each infinite}\,\,\mathbb M\subseteq\mathbb N\,\,\text{there are}\,\,a,b\,\,\text{with}\,\,0<a<b<2\pi\\
&\text{and}\,\,[a/m,b/m]\cap(K-t_0)\neq\emptyset\,\,\text{for infinitely many}\,\,m\in\mathbb M.
\end{split}
\end{equation}
If $f\in H(\mathbb D)\cap U(K)$, then the Taylor series of $f$ is not (C,1) summable at $t_0$.\newline
(b) Now let $K$ and $t_0$ be as in (a) and moreover let $K$ be symmetric with respect to $t_0$, i.e. $2t_0-t\in K$ for every $t\in K$.\newline
If $f\in L^1(\mathbb T)\cap U(K)$, then the Fourier series of $f$ is not (C,1) summable at $t_0$.  
\end{proposition}
\begin{proof}
(a) We denote by $\sum_{k=0}^{+\infty}c_kz^k$ the Taylor series of $f$ and by $S_n(f,t)=\sum_{k=0}^nc_ke^{ikt}$ the partial sums on $\mathbb T$ and let us 
assume that the Taylor series is (C,1) summable at $t_0$.
\par Since $f\in U(K)$, for every $h\in C(K)$ there is a subsequence $S_{n_j}(f,t)$ which converges to $h$ uniformly on $K$. Property (20) implies that there 
is a subsequence of $(n_j)$, which without loss of generality we may assume equal to $(n_j)$, and corresponding $\theta_j\in K-t_0$ such that $a\leq 
n_j\theta_j\leq b$ for all $j$.
\par Then Theorem (12.16) of \cite{Z}, Vol I, Ch III implies
$$S_{n_j}(f,t_0+\theta_j)-(S_{n_j}(f,t_0)-s)e^{in_j\theta_j}\to s,$$
where $s$ is the (C,1) sum of the Taylor series of $f$ at $t_0$.
\par Taking a further subsequence of $n_j$, we may assume that $n_j\theta_j\to\phi$ for some $\phi\in[a,b]$. Now the uniform convergence $S_{n_j}(f,t)\to 
h(t)$ on $K$ implies
$$h(t_0)-(h(t_0)-s)e^{i\phi}=s.$$
Choosing $h(t_0)\neq s$ we arrive at a contradiction.\newline
(b) The proof is the same as before. If $S_n(f,t)=\sum_{k=-n}^n\widehat f(k)e^{ikt}$ are the partial sums of the Fourier series of $f$, we assume that their 
(C,1) means at $t_0$ converge to some $s$. We consider $h\in C(K)$, the sequence $S_{n_j}(f,t)$ converging to $h$ uniformly on $K$ and the corresponding 
$\theta_j$ so that $\pm\theta_j\in K-t_0$ and $a\leq n_j\theta_j\leq b$ for all $j$.
\par We now apply Theorem (12.9) of \cite{Z}, Vol I, Ch III and get
$$\frac 12\big(S_{n_j}(f,t_0+\theta_j)+S_{n_j}(f,t_0-\theta_j)\big)-(S_{n_j}(f,t_0)-s)\cos n_j\theta_j\to s.$$
Exactly as before we get $h(t_0)-(h(t_0)-s)\cos\phi=s$ and we arrive at a contradiction choosing $h(t_0)\neq s$.
\end{proof}
\par We note that property (20) can be stated with ``left hand'' intervals $[t_0-(b/m),t_0-(a/m)]$ and the result of Proposition 4.1(a) remains unchanged.
\begin{corollary}
If $K$ is a compact proper subset of $\mathbb T$ which satifies property (20) for some $t_0\in K$ then $A(\mathbb D)\cap U(K)=\emptyset$. If moreover 
$K$ is symmetric with respect to $t_0$ then $C(\mathbb T)\cap U(K)=\emptyset$.\newline
In particular the one third Cantor set $C\subseteq[0,\pi]$ satisfies property (20) for $t_0=0\in C$ and hence $A(\mathbb D)\cap U(C)=\emptyset$. For the 
symmetric Cantor set $C^*=C\cup(-C)\subseteq[-\pi,\pi]$ we have that $C(\mathbb T)\cap U(C^*)=\emptyset$. 
\end{corollary}
\begin{proof}
Since the Fourier series of any $f\in A(\mathbb D)$ is (C,1) summable at every point of $\mathbb T$, the first statement is obvious.
\par Regarding the Cantor set $C$ and $t_0=0\in C$ we shall prove a stronger version of property (20): if $0<a<b<2\pi$ and $\frac ba>2$, then 
$[\frac an,\frac bn]\cap K\neq\emptyset$ for every $n\in\mathbb N$, $n\geq 3b$. 
\par We take $N\in\mathbb N$ so that $3^N\leq\frac nb<3^{N+1}$ and then $(\frac bn)/(\frac an)>2$ implies that the interval $(\frac an,\frac bn)$ contains at 
least one of the points $\frac 1{3^{N+1}}$ and $\frac 2{3^{N+1}}$ of $C$.
\end{proof}
\par In fact it is not very difficult to see that the one third Cantor set satisfies property (20) for every $t_0\in C$. 
\par The question whether $C(\mathbb T)\cap U(C)=\emptyset$ is true is only slightly more complicated. The one third Cantor set $C\subseteq[0,\pi]$ is not 
symmetric with respect to $0$ and in fact it is not symmetric with respect to any of its points. Nevertheless $C(\mathbb T)\cap U(C)=\emptyset$ is true.
\begin{proposition}
Let $K$ be a compact proper subset of $\mathbb T$ and $t_0\in K$ with the following property:
\begin{equation}
\begin{split}
&\text{for each infinite}\,\,\mathbb M\subseteq\mathbb N\,\,\text{there are}\,\,t_m\in K\,\,\text{for all}\,\,m\in\mathbb M\\
&\text{so that}\,\,t_m\to t_0\,\,\text{and there are}\,\,a,b\,\,\text{with}\,\,0<a<b<2\pi\\
&\text{and}\,\,[a/m,b/m]\cap(K-t_m)\cap(t_m-K)\neq\emptyset\\
&\text{for infinitely many}\,\,m\in\mathbb M.
\end{split}
\end{equation}
If $f\in L^1(\mathbb T)\cap U(K)$, then the (C,1) means of the Fourier series of $f$ do not converge uniformly on $K$.   
\end{proposition}
\begin{proof}
The proof is a variation of the proof of Proposition 4.1.
\par Let $S_n(f,t)=\sum_{k=-n}^n\widehat f(k)e^{ikt}$ be the partial sums of the Fourier series of $f$, and assume that their (C,1) means converge to some 
{\it function} $s$ uniformly on $K$. We consider an arbitrary $h\in C(K)$ and a sequence $S_{n_j}(f,t)$ converging to $h$ uniformly on $K$.
\par Now there are $t_j\in K$ so that $t_j\to t_0$ and (considering a subsequence) corresponding $\theta_j$ so that $t_j\pm\theta_j\in K$ and $a\leq 
n_j\theta_j\leq b$ for all $j$.
\par We apply again Theorem (12.9) of \cite{Z}, Vol I, Ch III and get
$$\frac 12\big(S_{n_j}(f,t+\theta_j)+S_{n_j}(f,t-\theta_j)\big)-(S_{n_j}(f,t)-s(t))\cos n_j\theta_j\to s(t)$$
uniformly for $t\in K$. In fact here we apply the proof rather than the statement of the above Theorem (12.9) of \cite{Z}. One should read the small paragraph a 
few lines before the statement of Theorem (12.6): ``we also observe that if the terms of $u_0+u_1+\ldots$\, depend on a parameter, and if the hypotheses 
concerning this series are satisfied uniformly, the conclusions also hold uniformly''. The relevant hypothesis in our case is the uniform (C,1) summability of 
the Fourier series of $f$. 
\par Since $t_j\to t_0$ and $t_j\pm\theta_j\to t_0$ and since we may assume that $n_j\theta_j\to\phi$ for some $\phi\in[a,b]$, we finally get 
$$h(t_0)-(h(t_0)-s)e^{i\phi}=s$$
and, choosing $h(t_0)\neq s$, we arrive at a contradiction.
\end{proof}
\begin{corollary}
If $K$ is a compact proper subset of $\mathbb T$ which satifies property (21) for some $t_0\in K$ then $C(\mathbb T)\cap U(K)=\emptyset$.\newline
In particular the one third Cantor set $C\subseteq[0,\pi]$ satisfies property (21) for every $t_0\in C$ and hence $C(\mathbb T)\cap U(C)=\emptyset$.
\end{corollary}
\begin{proof}
For the first statement we recall that if $f\in C(\mathbb T)$ then the (C,1) means of the Fourier series of $f$ converge to $f$ uniformly on $\mathbb T$.
\par The one third Cantor set satisfies property (21) in a slightly stronger form: there are $a,b$ with $0<a<b<2\pi$ so that for every $t_0\in C$ there are 
$t_n\in C$ for all $n\in\mathbb N$ with $t_n\to t_0$ and $[\frac an,\frac bn]\cap(K-t_n)\cap(t_n-K)\neq\emptyset$ for every $n\in\mathbb N$.
\par In fact, take arbitrary $a,b$ with $0<a<b<2\pi$ and $\frac ba>3$ and any $t_0\in C$. For each $N\in\mathbb N$ we consider the interval $I_N$ of length 
$\frac 1{3^N}$ which appears at the $N$-th step of the construction of $C$ and which contains $t_0$. If $I_{N,-}$ and $I_{N,+}$ are the left and right 
subintervals of $I_N$ of length $\frac 1{3^{N+1}}$, then $t_0$ belongs to one of them (and this is $I_{N+1}$). If $t_0\in I_{N,-}$ then we define $t_n$ to be 
the right endpoint of $I_{N,-}$ for every $n$ with $3^Nb\leq n<3^{N+1}b$. If $t_0\in I_{N,+}$ then we define $t_n$ to be the left endpoint of $I_{N,+}$ for 
every $n$ with $3^Nb\leq n<3^{N+1}b$. In the case $t_0\in I_{N,-}$ the three points $t_n$ and $t_n\pm\frac 1{3^{N+1}}$ belong to $C$ and $\frac 
1{3^{N+1}}\in(\frac an,\frac bn)$ since $3^{N+1}a<3^Nb\leq n<3^{N+1}b$. The case $t_0\in I_{N,+}$ is similar. 
\end{proof}
\par It is not hard to construct countably infinite sets satistfying (20). Take for example $K=\{0\}\cup\{\frac 1n\,|\,n\in\mathbb N\}$. Hence,
\begin{corollary}
There are countably infinite compact $K\subseteq\mathbb T$ such that $A(\mathbb D)\cap U(K)=\emptyset$.  
\end{corollary}
\par We close with the following open problems.
\par I. {\it Construct a perfect set $K$ and a function $f\in A(\mathbb D)\cap U(K)$}.
\par II. {\it Study pointwise universality on uncountable sets $E$ with Lebesgue measure zero for functions in $C(\mathbb T)$ or $A(\mathbb D)$}.\newline
Of course the functions $g$ in Definition 1.1 of pointwise universality must belong to the first class of Baire.
\par III. {\it Is it true that, if $E$ has positive Hausdorff dimension, then $E$ does not accept universality?}
\par {\bf Acknowledgement}. We would like to thank the referee for valuable comments which cleared certain subtle points of the proofs and helped us improve the 
presentation of this paper. Proposition 4.1 in our original paper contained nothing about sets $K$ symmetric with respect to $t_0$ and hence our original 
Corollary 4.1 considered only the result $A(\mathbb D)\cap U(C)=\emptyset$. The referee asked whether we could prove that $C(\mathbb T)\cap U(C)=\emptyset$ and 
this prompted us to work first with symmetric sets $K$ and get that $C(\mathbb T)\cap U(C^*)=\emptyset$. After a while we proved the stronger $C(\mathbb T)\cap 
U(C)=\emptyset$. We include the weaker result since it shows that we do not need extra hypotheses (i.e. the uniform (C,1) summability) in the presence of the 
symmetry of the set $K$.


\begin{thebibliography}{99}

\bibitem{BH}
F. Bayart \and Y. Heurteaux, {\it Multifractal analysis for the divergence of Fourier series}. arXiv:1101.3027

\bibitem{HK}
G. Herzog \and P. C. Kunstmann, {\it Universally divergent Fourier series via Landau's extremal functions}. Comment. Math. Univ. Carolinae {\bf 56} (2015), 
159--168.

\bibitem{Ka}
Y. Katznelson, {\it Introduction to Harmonic Analysis}. Cambridge Univ. Press, Cambridge, 3rd edition, 2004.

\bibitem{K}
J. P. Kahane, {\it Baire's category theorem and trigonometric series}. J. Anal. Math. \textbf{80} (2000), 143--182.

\bibitem{KNP}
E. Katsoprinakis, V. Nestoridis \and C. Papachristodoulos, {\it Universality and Cesaro summability}. Comp. Methods and Function Theory \textbf{12} (2012),
419--448.

\bibitem{Ko}
T. W. K\"orner, {\it Kahane's Helson curve}. Proccedings of the Conference in Honor of J. P. Kahane, Orsay, 1993, J. Fourier Anal. Appl. (1995), 325--346.

\bibitem{M}
J. M\"uller, {\it Continuous functions with universally divergent Fourier series on small subsets of the circle}. C. R. Acad. Sci. Paris, Ser. I \textbf{348}
(2010), 1155--1158.

\bibitem{N}
V. Nestoridis, {\it Universal Taylor series}. Ann. Inst. Fourier {\bf 46}, no 5 (1996), 1293--1306. 

\bibitem{Ox}
J. Oxtoby, {\it Measure and Category}. Springer-Verlag Inc. New York, 2nd edition, 1980.

\bibitem{Z}
A. Zygmund, {\it Trigonometric Series, Vol. I}. Cambridge Univ. Press, Cambridge, 3rd edition, 2002. 

\end{thebibliography}
\end{document}